\documentclass[a4paper,12pt,leqno]{amsart}
\usepackage{amsmath,amsfonts,amsthm,amssymb,dsfont}
\usepackage[alphabetic]{amsrefs}
\usepackage[OT4]{fontenc}
\usepackage{mathrsfs}
\usepackage{enumerate, xspace}
\usepackage[pdftex]{graphicx}
\usepackage{graphics}

\long\def\symbolfootnote[#1]#2{\begingroup
\def\thefootnote{\fnsymbol{footnote}}\footnote[#1]{#2}\endgroup}

\newtheorem{thm}{Theorem}[section]
\newtheorem{cor}[thm]{Corollary}
\newtheorem{sublem}[thm]{Sublemma}
\newtheorem{lemma}[thm]{Lemma}
\newtheorem{prop}[thm]{Proposition}
\newtheorem*{claim}{Claim}

\theoremstyle{definition}
\newtheorem{defin}[thm]{Definition}
\newtheorem{rem}[thm]{Remark}
\newtheorem*{rema}{Remark}
\newtheorem{example}[thm]{Example}

\newcommand{\centra}{\mathscr{Z}}
\newcommand{\la}{\langle}
\newcommand{\ra}{\rangle}
\newcommand{\FF}{\mathbf F}
\newcommand{\A}{\mathbb A}
\newcommand{\N}{\mathbf N}

\title{Bipolar Coxeter groups}
\author{}

\begin{document}
\maketitle

\begin{center}\bf
Pierre-Emmanuel Caprace$^a$\symbolfootnote[1]{Supported by the Belgian Fund for Scientific Research (FNRS).}
\& Piotr Przytycki$^b$\symbolfootnote[2]{Partially supported by
MNiSW grant N201 012 32/0718 and the Foundation for Polish Science.}
\end{center}

\begin{center}\it
$^a$ Universit\'e catholique de Louvain, D\'epartement de Math\'ematiques, \\
Chemin du Cyclotron 2, 1348 Louvain-la-Neuve, Belgium\\
\emph{e-mail:} \texttt{pe.caprace@uclouvain.be}
\end{center}

\begin{center}\it
$^b$ Institute of Mathematics, Polish Academy of Sciences,\\
\'Sniadeckich 8, 00-956 Warsaw, Poland\\
\emph{e-mail:} \texttt{pprzytyc@mimuw.edu.pl}
\end{center}

\begin{abstract}
\noindent
We consider the class of those Coxeter groups for which removing from the Cayley graph any tubular neighbourhood of any wall leaves exactly two connected components. We call these Coxeter groups \emph{bipolar}. They include both the virtually Poincar\'e duality Coxeter groups and the infinite irreducible $2$-spherical ones. We show in a geometric way that a bipolar Coxeter group admits a unique conjugacy class of Coxeter generating sets. Moreover, we provide a characterisation of bipolar Coxeter groups in terms of the associated Coxeter diagram.
\end{abstract}

%%%%%%%%%%%%%%%%%%%%%%%%%%%%%%%%%%%%%%%%%%%%%%%%
%%%%%%%%%%%%%%%%%%%%%%%%%%%%%%%%%%%%%%%%%%%%%%%%
\section{Introduction}
%%%%%%%%%%%%%%%%%%%%%%%%%%%%%%%%%%%%%%%%%%%%%%%%
%%%%%%%%%%%%%%%%%%%%%%%%%%%%%%%%%%%%%%%%%%%%%%%%

Much of the algebraic structure of a Coxeter group is determined
by the combinatorics of the walls and half-spaces of the
associated Cayley graph (or Davis complex). When investigating
rigidity properties of Coxeter groups, it is therefore natural to
consider the class of Coxeter groups whose half-spaces are
well-defined up to quasi-isometry. This motivates the following
definition.

Let $W$ be a finitely generated Coxeter group. Fix a Coxeter
generating set $S$ for $W$. Let $X$ denote the Cayley graph
associated with the pair $(W, S)$. An element $s\in S$ is called
\textbf{bipolar} if any tubular neighbourhood of the $s$-invariant
wall $\mathcal W_s$ separates $X$ into exactly two connected
components. In fact, we shall later give an alternative
Definition~\ref{defin:bipolar} and prove equivalence with this one
in Lemma~\ref{lem:clarif}. Another equivalent condition is
$$\tilde e(W, \centra_W(s)) = 2,$$
where $\tilde e(\cdot, \cdot)$ is the quasi-isometry invariant introduced by Kropholler and Roller in \cite{KR}. See Appendix~\ref{App} for details.

We further say that $W$ is \textbf{bipolar} if it admits some Coxeter generating set all of whose elements are bipolar. We will prove, in Corollary~\ref{cor:refl}, that if $W$ is bipolar, then every Coxeter generating set consists of bipolar elements.

A basic class of examples of bipolar Coxeter groups is provided by the following.

\begin{prop}\label{prop:manifold}
A Coxeter group which admits a proper and cocompact action on a contractible manifold is bipolar.
\end{prop}

\begin{proof}
The Coxeter group $W$ in question is a virtual Poincar\'e duality group of dimension $n$. By \cite[Corollary~5.6]{Davis}, for each $s\in S$ its centraliser $\centra_W(s)$ is a virtual Poincar\'e duality group of dimension $n-1$. Then, in view of \cite[Corollary~4.3]{KR}, there is a finite index subgroup $W_0$ of $W$ satisfying $\tilde e(W_0, W_0\cap \centra_W(s)) = 2$.
Using \cite[Lemma 2.4(iii)]{KR} we then also have $\tilde e(W, \centra_W(s)) = 2$. By Lemma~\ref{lem:tildeecoincides} below this means that $s$ is bipolar, as desired.
\end{proof}

No purely combinatorial criterion in terms of the Coxeter diagram seems to be known to decide whether a given Coxeter group acts properly and cocompactly on a contractible manifold. On the other hand, the following result provides a characterisation of bipolarity in terms of the Coxeter diagram.

All the relevant notions are recalled in Section~\ref{sec:preliminaries} below. The only less standard terminology is that we call two elements $s, s'$ of some Coxeter generating set $S$ \textbf{odd-adjacent} if the order of $ss'$ is finite and odd. This turns $S$ into the vertex set of a graph whose connected components are called the \textbf{odd components} of $S$.

\begin{thm}\label{thm:bipolar}
A finitely generated Coxeter group $W$ is bipolar if and only if it
admits some Coxeter generating set $S$ satisfying the following three
conditions.
\begin{enumerate}[(a)]
\item There is no spherical irreducible component $T$ of $S$.

\item There are no $I\subset T$ with $T$ irreducible and $I$ non-empty spherical such that $I\cup T^\perp$ separates some vertices of the Coxeter diagram of $S$.

\item If $T \subset S$ is irreducible spherical and an odd component $O$ of $S$ is contained in $T^\perp$, then there are adjacent $t \in O$ and $t' \in S \setminus (T \cup T^\perp)$.
\end{enumerate}
\end{thm}

\begin{cor}\label{cor:2sph}
Any infinite irreducible $2$-spherical Coxeter group is bipolar.
\end{cor}

Bipolarity is thus a condition which is naturally shared by both
infinite irreducible $2$-spherical Coxeter groups and virtually
Poincar\'e duality Coxeter groups. By the works of Charney--Davis
\cite{CharneyDavis}, Franzsen--Howlett--M{\"u}hlherr \cite{FHM},
and Caprace--M\"uhlherr \cite{CM} the Coxeter groups in those two
classes are rigid in the sense that they admit a unique conjugacy
class of Coxeter generating sets. The following result shows that
this property is in fact shared by all bipolar Coxeter groups.

\begin{thm}\label{thm:CoxGenSet}
In a bipolar Coxeter group, any two Coxeter generating sets are conjugate.
\end{thm}

Before we discuss this result, we give an immediate corollary. A \textbf{graph automorphism} of a Coxeter group is an automorphism which permutes the elements of a given Coxeter generating set, and thus corresponds to an automorphism of the associated Coxeter diagram. An automorphism of a Coxeter group is called \textbf{inner-by-graph} if it is a product of an inner automorphism and a graph automorphism.

\begin{cor}\label{cor:Inner-by-graph}
Every automorphism of a bipolar Coxeter group is inner-by-graph.
\end{cor}

Theorem~\ref{thm:CoxGenSet} both generalises and unifies the main
results of \cite{CharneyDavis}, \cite{CM} and \cite{FHM}. The
proof we shall provide is self-contained and based on the fact
that the bipolar condition makes the half-spaces into a coarse
notion which is preserved under quasi-isometries coming from
changing the generating set.

Theorem~\ref{thm:CoxGenSet} resulted from an attempt to find a
geometric property of so called \emph{twist-rigid} Coxeter groups
which would provide an alternative proof of the following, which
is the main result from \cite{TwistRigid}.

\begin{thm}[{\cite[Theorem 1.1 and Corollary 1.3(i)]{TwistRigid}}]
\label{thm:twistrigid}
In a twist-rigid Coxeter group, any two angle-compatible Coxeter generating sets are conjugate.
\end{thm}

However, by Theorem~\ref{thm:bipolar} many twist-rigid Coxeter
groups are not bipolar, hence one cannot use
Theorem~\ref{thm:CoxGenSet} to deduce
Theorem~\ref{thm:twistrigid}. On the other hand, a combination of
Theorems~\ref{thm:twistrigid}~and~\ref{thm:bipolar} together with
the main results of \cite{HM} and \cite{MM} yields
Theorem~\ref{thm:CoxGenSet}. Despite of this fact, we believe that
the direct geometric proof we provide here sheds some light on
existing rigidity results on Coxeter groups.
Note for example that the proof of Theorem~\ref{thm:twistrigid}
which we give in \cite{TwistRigid} relies on the fact that in an
infinite irreducible $2$-spherical Coxeter group all Coxeter
generating sets are conjugate.

\medskip

The article is organised as follows. In Section~\ref{sec:Coxeter}
we collect some basic facts on Coxeter groups. In
Section~\ref{sec:rigidity section} we discuss properties of
bipolar Coxeter groups and prove Theorem~\ref{thm:CoxGenSet}. In
Section~\ref{sec:CharNearlyBiv} we characterise \emph{nearly
bipolar reflections}, which are reflections enjoying significant
geometric properties slightly weaker than the ones of bipolar
reflections. Finally, in Section~\ref{sec:bipolar reflections} we
characterise bipolar reflections and prove
Theorem~\ref{thm:bipolar}. In Appendix~\ref{App} we give a survey
on different approaches to the notion of \emph{poles}.

\subsection*{Acknowledgements} We thank Hausdorff Research Institute for Mathematics in Bonn and Erwin Schr\"odinger International Insitute for Mathematical Physics in Vienna, where the article was written. Special thanks are due to Michah Sageev for drawing our attention to \cite{KR}.

%%%%%%%%%%%%%%%%%%%%%%%%%%%%%%%%%%%%%%%%%%%%%%%%
%%%%%%%%%%%%%%%%%%%%%%%%%%%%%%%%%%%%%%%%%%%%%%%%
\section{Coxeter groups}
\label{sec:Coxeter}
%%%%%%%%%%%%%%%%%%%%%%%%%%%%%%%%%%%%%%%%%%%%%%%%
%%%%%%%%%%%%%%%%%%%%%%%%%%%%%%%%%%%%%%%%%%%%%%%%
\subsection{Preliminaries}
\label{sec:preliminaries}

Let $W$ be a finitely generated Coxeter group and let $S\subset W$ be a Coxeter generating set. We start with explaining the notions appearing in the statement of Theorem~\ref{thm:bipolar}.

Given a subset $J \subset S$, we set $W_J = \la J \ra$.
We say that $W_J$ is \textbf{spherical} if it is finite. The
subset $J$ is called \textbf{spherical} if $W_J$ is spherical. It
is called \textbf{$2$-spherical} if all of its two-element subsets
are spherical. Two elements of $S$ are called \textbf{adjacent} if
they form a spherical pair. This defines a graph with vertex set
$S$ which is called the \textbf{Coxeter graph}. We emphasize that
this terminology is not standard; for us a Coxeter graph is not a
labelled graph; the non-edges correspond to pairs of generators
generating an infinite dihedral group. In this terminology $J$ is
$2$-spherical if its Coxeter graph is a complete graph. A Coxeter
group is \textbf{$2$-spherical} if it admits a Coxeter generating
set $S$ which is $2$-spherical. A \textbf{path} in $S$ is a
sequence in $S$ whose consecutive elements are adjacent.

We denote by $J^\perp$ the subset of $S\setminus J$ consisting of
all elements commuting with all the elements of $J$. A subset $J
\subset S$ is called \textbf{irreducible} if it is not contained
in $K\cup K^\perp$ for some non-empty proper subset $K\subset J$.
An \textbf{irreducible component} of $s\in S$ in $J\subset S$ is
the maximal irreducible subset of $J$ containing $s$. If $J$
satisfies $S=J\cup J^\perp$, then $W_J$ is called a
\textbf{factor} of $W$.

\medskip

The Cayley graph associated with the pair $(W,S)$ with the path-metric in which the edges have length $1$ is denoted by $(X,d)$. The corresponding Davis complex is denoted by $\mathbb A$. A \textbf{reflection} is an element of $W$ conjugate to an element of $S$.
Given a reflection $r \in W$, we denote by $\mathcal W_r$ its fixed-point set in $X$, the \textbf{wall} associated with $r$. We use the notation $\mathcal W^\A_r$ for the fixed point set of $r$ in $\A$.
The two connected components of the complement of a wall are called \textbf{half-spaces}.
We say that two walls $\mathcal W_{r_1}, \mathcal W_{r_2}$ \textbf{intersect} if the corresponding $\mathcal W^\A_{r_1}, \mathcal W^\A_{r_2}$ intersect, \emph{i.e.} if $r_1r_2$ is of finite order.
The walls $\mathcal W_{r_1}, \mathcal W_{r_2}$ are \textbf{orthogonal}, if $r_1$ commutes with and is distinct from $r_2$.

A \textbf{parabolic} subgroup $P\subset W$ is a subgroup conjugate
to $W_T$ for some $T\subset S$. Any $P$-invariant translate of the
Cayley graph of $W_T$ in $X$ is called a \textbf{residue} of $P$.

If $v$ is a vertex of $X$ and $w$ is an element of $W$, we denote by $w.v$ the translate of $v$ in $X$ under the action of $w$.
\medskip

We will need some additional non-standard notation. Let $v$ be a
vertex of $X$. We say that $v$ is \textbf{adjacent} to a wall
$\mathcal W$ if the distance from $v$ to $\mathcal{W}$ equals
$\frac{1}{2}$. We denote by $S_v$ the set of all reflections with
walls adjacent to $v$. Thus $S_v$ is a Coxeter generating set
conjugate to $S$ \emph{via} the element mapping the identity
vertex to $v$. In particular, if $v$ is the identity vertex $v_0$, then
we have $S_{v_0}=S$. 
%PE added the following
Moreover generally, for any vertex $v$, there is a canonical bijection between $S_v$ and $S$ which is realised by the conjugation under the unique element of $W$ mapping $v$ to $v_0$. 
We say that a subset of $S_v$ is
\textbf{spherical, irreducible}, \emph{etc.}, if its conjugate in
$S$ is so. In particular, for $T\subset S_v$ we denote by
$T^\perp$ the subset of $S_v\setminus T$ consisting of elements
commuting with all the elements of $T$. Similarly, for $T\subset
S_v$ we denote $W_T=\la T \ra$. Note that in case $S_v=T\cup
T^\perp$ the parabolic subgroup $W_T$ is a conjugate of a factor
of $W$.

Let now $r$ be a reflection in $W$. We denote by $T_{v,r}$ the smallest subset of $S_{v}$ satisfying $r \in \la T_{v,r} \ra$. This set should be thought of as the \emph{support} of $r$ with respect to $S_v$.

By $J_{v, r}$ we denote the subset of $S_v$ consisting of elements $s$ satisfying  $d(s.v, \mathcal W_r)<d(v, \mathcal W_r)$ or $s=r$.
Observe that we have $J_{v,r}\subset T_{v,r}$.

Finally, let $U_{v, r}$ be the set of elements of $S_v$ commuting with $r$, but different from $r$. Equivalently (see \cite[Lemma 1.7]{BH}), $s$ belongs to $U_{v,r}$ if it satisfies $d(s.v, \mathcal W_r)=d(v, \mathcal W_r)$ and $s\neq r$. In particular $U_{v,r}$ is disjoint from $J_{v,r}$. We also have $T_{v, r}^\perp \subset U_{v, r}$. On the other hand, an easy computation shows
$$U_{v, r} \subset T_{v, r} \cup T_{v, r}^\perp.$$
We also have the following basic fact.

\begin{lemma}[{\cite[Lemma~8.2]{TwistRigid}}]\label{lem:shadow}
For any vertex $v$ of $X$ and any reflection $r$, the set $J_{v, r} \cup (U_{v, r}\cap T_{v, r})$ is spherical.
\end{lemma}

We deduce a useful corollary.

\begin{cor}\label{lem:SphFactor}
Let $r \in W$ be a reflection not contained in a conjugate of any
spherical factor of $W$. Then every vertex $v$ of $X$ is adjacent
to some vertex $v'$ satisfying $d(v', \mathcal W_r) > d(v,
\mathcal W_r)$.
\end{cor}

\begin{proof}
Set $T =T_{v,r}$, $J = J_{v,r}$, and $U = U_{v, r}$. Suppose, by
contradiction, that for each vertex $v'$ adjacent to $v$ we have
$d(v', \mathcal W_r) \leq d(v, \mathcal W_r)$. This means that we
have $S_v = J \cup U$. From $J \subset T$ we deduce $T = J \cup (U
\cap T)$. Furthermore, by Lemma~\ref{lem:shadow} the set $J\cup (U
\cap T)$ is spherical. Thus $W_T$ contains $r$ and is conjugate to
a spherical factor of $W$. Contradiction.
\end{proof}

\subsection{Parallel Wall Theorem}

We now discuss the so-called \emph{Parallel Wall Theorem}, first established by Brink and Howlett~\cite[Theorem 2.8]{BH}. The theorem stipulates the existence of a constant $L$ such that for any wall $\mathcal W$ and any vertex $v$ at distance at least $L$ from $\mathcal W$, there is another wall separating $v$ from $\mathcal W$. The following strengthening of this fact is established (implicitly)
in~\cite[Section 5.4]{Cap_AGT}.

\begin{thm}[Strong Parallel Wall Theorem]
\label{parallel wall theorem main} For each $n$ there is a constant $L$ such that for any wall $\mathcal W$ and any vertex $v$ in the Cayley graph $X$ at distance at least $L$ from $\mathcal W$, there are at least $n$ pairwise parallel walls separating $v$ from $\mathcal W$.
\end{thm}

In order to state a corollary we need to define \emph{tubular neighbourhoods}. Given a metric space $(X,d)$ and a subset $H \subset X$, we denote
$$\mathcal N^X_k(H) = \{ x \in X \; | \; d(x, H) \leq k\}.$$
We call this set the \textbf{$k$-neighbourhood} of $H$. A \textbf{tubular neighbourhood} of $H$ is a $k$-neighbourhood for some $k>0$ (usually we consider only $k\in \N$). We record an immediate consequence of Theorem~\ref{parallel wall theorem main}.

\begin{cor}\label{cor:ParallelWall}
For each $k\in \N$ there is a constant $L$ such that for any wall $\mathcal W$ and any vertex $v$ of $X$ at distance at least $L$ from $\mathcal W$, there is another wall separating $v$ from the tubular neighbourhood  $\mathcal N^X_k(\mathcal W)$.
\end{cor}

\subsection{Complements of tubular neighbourhoods of walls}

We need the following result about the complements of tubular neighbourhoods of walls and their intersections.
We denote by $\partial \phi$ the boundary wall of a half-space $\phi$ in the Cayley graph $X$.

\begin{lemma}\label{lem:Fuchsian}
Assume that $(W,S)$ has no spherical factor. Let $k\in \N$.
\begin{enumerate}[(i)]
\item For each half-space $\phi$, the set $\phi \setminus \mathcal
N^X_k(\partial \phi)$ is non-empty.
\item Let $\phi , \phi'$ be a
pair of non-complementary half-spaces whose walls $\partial \phi,
\partial \phi' $ intersect. Then the intersection $(\phi \setminus
\mathcal N^X_k(\partial \phi)) \cap (\phi' \setminus \mathcal
N^X_k(\partial \phi'))$ is also non-empty.
\item Let $\phi,\phi'$
be a pair of non-complementary half-spaces with $\partial
\phi\subset \phi',\ \partial \phi'\subset \phi$, whose associated
pair of reflections $r,r'$ belongs to the Coxeter generating set
$S$. Assume additionally that $\{r,r'\}$ is not an irreducible
factor of $S$. Then the intersection $(\phi \setminus \mathcal
N^X_k(\partial \phi)) \cap (\phi' \setminus \mathcal
N^X_k(\partial \phi'))$ is non-empty.
\end{enumerate}
\end{lemma}

In assertion (iii) we could relax the hypothesis to allow any
intersecting half-spaces bounded by disjoint walls. But then we
have to additionally assume that the corresponding reflections do
not lie in an affine factor of $W$. We will not need this in the
article.

\begin{proof}
(i) This follows directly from Corollary~\ref{lem:SphFactor}.

(ii) Denote by $r,r'$ the reflections in $\partial\phi,\partial
\phi'$. By (i) and the Parallel Wall Theorem, there is a
reflection $t\in W$ such that $tr$ is of infinite order (for
another argument, see \emph{e.g.} \cite[Proposition~8.1]{Hee}). If
$r$ does not commute with $r'$, then by \cite{Deodhar_refl} the
reflection group $\la r, r', t\ra $ is an infinite irreducible
rank-$3$ Coxeter group (affine or hyperbolic). It remains to
observe that the desired property holds in the special case of
affine and hyperbolic triangle groups.

If $r$ commutes with $r'$, we take a vertex $v$ in $\phi\cap \phi'$. By Corollary~\ref{lem:SphFactor} there is a path of length $2k$ from $v$ to some $v'$ such that each consecutive vertex is farther from $\partial \phi$. Then $d(v',\partial \phi)$ is at least $2k+\frac{1}{2}$. Since  $\partial\phi$ and $\partial \phi'$ are orthogonal, this path stays in $\phi'$. Similarly, there is a path of length $k$ from $v'$ to some $v''$ such that each consecutive vertex is farther from $\partial \phi'$. Then $v''$ lies in $(\phi \setminus \mathcal N^X_k(\partial \phi)) \cap (\phi' \setminus \mathcal N^X_k(\partial \phi'))$.

(iii) Since $\{r,r'\}$ is not an irreducible component of $S$,
there is an element $s\in S$ which does not commute with one of
$r$ and $r'$. Hence the parabolic subgroup $\la r, r', s\ra $ is a
hyperbolic triangle group. As before we observe that the desired
property holds in this special case.
\end{proof}

\subsection{Position of rank-$2$ residues}

We conclude with the discussion of the possible positions of a rank-$2$ residue with respect to a wall.

\begin{lemma}\label{lem:polygon}
Let $r \in W$ be a reflection and let $R \subset X$ be a
residue of rank $2$ containing a vertex $v$. Assume that the vertices $x,y$ adjacent to $v$ in $R$ satisfy $d(v,\mathcal W_r)<d(y,\mathcal W_r)$ and $d(v,\mathcal W_r)\leq d(x,\mathcal W_r)$. Then for any vertex $z$ in $R$ we have
\begin{equation*}
 d(z,\mathcal W_r)= d(v,\mathcal W_r)+
 \begin{cases}
 d(z,v) & \text{if } d(v,\mathcal W_r)< d(x,\mathcal W_r),
\\
d(z,\{v,x\}) & \text{if } d(v,\mathcal W_r)= d(x,\mathcal W_r).
\end{cases}
\end{equation*}
\end{lemma}

In particular, no wall orthogonal to $\mathcal W_r$ crosses $R$, except possibly for the one adjacent to $v$ and $x$.

The proof is a simple calculation using \emph{root systems} (see \emph{e.g.} \cite{BH}) and will be omitted.

%%%%%%%%%%%%%%%%%%%%%%%%%%%%%%%%%%%%%%%%%%%%%%%%
%%%%%%%%%%%%%%%%%%%%%%%%%%%%%%%%%%%%%%%%%%%%%%%%
\section{Rigidity of bipolar Coxeter groups}
%%%%%%%%%%%%%%%%%%%%%%%%%%%%%%%%%%%%%%%%%%%%%%%%
%%%%%%%%%%%%%%%%%%%%%%%%%%%%%%%%%%%%%%%%%%%%%%%%
\label{sec:rigidity section}

In this section we define \emph{bipolar} Coxeter groups and prove that this definition agrees with the one given in the Introduction (Lemma~\ref{lem:clarif}). Then we prove that in a bipolar Coxeter group all Coxeter generating sets are \emph{reflection-compatible} (Corollary~\ref{cor:refl}). We conclude with the proof of Theorem~\ref{thm:CoxGenSet}.

\subsection{Bipolar Coxeter groups}
\label{sec:bipolar Coxeter groups}

Let $G$ be a finitely generated group and let $X$ denote the Cayley graph associated with some finite generating set for $G$. We view $X$ as a metric space with the path-metric obtained by giving each edge length~$1$. We identify $G$ with the $0$-skeleton $X^{(0)}$ of $X$. Let $H$ be a subset of $G$.

\begin{defin}
A \textbf{pole} (in $X$) of $G$ relative to $H$ (or of the pair $(G, H)$) is a chain of the form $U_1 \supset U_2 \supset \dots$, where $U_k$ is a non-empty connected component of $X \setminus \mathcal N^X_k(H)$.
\end{defin}

In the appendix, different equivalent definitions of poles as well
as their basic properties will be discussed. Here we merely record
that in Lemma~\ref{lem:quasiisometry} we show that there is a
correspondence between the collections of poles of the pair
$(G,H)$ determined by different generating sets. Hence it makes
sense to consider the number of poles $\tilde e(G,H)$ as an
invariant of the pair $(G,H)$.

We say that the pair $(G, H)$ (or simply the subset $H$ when there
is no ambiguity on what the ambient group is) is
\textbf{$n$-polar} if we have $\tilde e(G, H)= n$. We shall mostly
be interested in the case $n=2$, in which case we say that $H$ is
\textbf{bipolar}. In case $n=1$ we say that $H$ is
\textbf{unipolar}. Notice that $G$ has $n$ ends if and only if the
trivial subgroup is $n$-polar.

\begin{defin}\label{defin:bipolar}
A generator $s$ in some Coxeter generating set $S$ of a Coxeter group $W$ is called \textbf{bipolar} if its centraliser $\centra_W(s)$ is so. The group $W$ is \textbf{bipolar} if it admits a Coxeter generating set all of whose elements are bipolar.
\end{defin}

We now verify that this definition agrees with the one given in the Introduction.

\begin{lemma}\label{lem:clarif}
A generator $s\in S$ is bipolar if and only if $X \setminus
\mathcal N^X_{k}(\mathcal W_s)$ has exactly two connected
components for any $k\in \N$.
\end{lemma}

Before we can give the proof we need the following discussion.
\begin{rem}
\label{rem:centraliser is the wall}
The centraliser $\centra_W(s)$ coincides with the stabiliser of $\mathcal W_s$ in the Cayley graph $X$.
Since the action of $W$ on $X$ has only finitely many orbits of edges, it follows that $\centra_W(s)$ acts cocompactly on the associated wall $\mathcal W_s$. Hence $\mathcal W_s$ is at finite Hausdorff distance in $X$ from $\centra _W(s)\subset X^{(0)}$. Thus, by Remark~\ref{rem:olddefin}, $\tilde e(W, \centra_W(s))$ is equal to the number of poles of $(X,{\mathcal W_s})$ (see Appendix~\ref{App}).
\end{rem}

\begin{lemma}\label{lem:Reflection:basic}
Let $r$ be a reflection in $W$. Then
\begin{enumerate}[(i)]
\item $r$ is not unipolar, \item moreover we have $\tilde e(W,
\centra_W(r)) = 0$ if and only if $r$ belongs to a conjugate of
some spherical factor of $W$.
\end{enumerate}
\end{lemma}

\begin{proof}
(i) By Remark~\ref{rem:centraliser is the wall} we need to study
the poles of $(X,{\mathcal W_r})$. Since $r$ acts non-trivially on
the two components of $X \setminus \mathcal W_r$ it follows that
the number of poles of $(X, \mathcal W_r)$ is even (or infinite).

(ii) If $r$ belongs to a conjugate of some spherical factor of
$W$, then $\centra_W(r)$ has finite index in $W$ and hence  we
have $\tilde e(W, \centra_W(r))=0$. Conversely, assume that $s$
does not belong to a conjugate of any spherical factor of $W$.
Then Corollary~\ref{lem:SphFactor} ensures that $X$ does not
coincide with any tubular neighbourhood of $\mathcal W_r$, hence
we have $\tilde e(W, \centra_W(r)) \neq 0$.
\end{proof}

We are now prepared for the following.

\begin{proof}[Proof of Lemma~\ref{lem:clarif}]
First assume that $X \setminus \mathcal N^X_{k}(\mathcal W_s)$ has
exactly two connected components for any $k\in \N$. Since these
components are interchanged under the action of $s$, they are
either both contained or neither of them is contained in a tubular
neighbourhood of $\mathcal W_s$. In fact, since the hypothesis is
satisfied for every $k$, neither of them is contained in a tubular
neighbourhood of $\mathcal W_s$. Hence they determine the only two
poles of $(X,\mathcal W_s)$.
Then $s$ is bipolar by Remark~\ref{rem:centraliser is the wall}.

For the converse, let $s$ be bipolar. Like before, by Remark~\ref{rem:centraliser is the wall} the pair $(X,\mathcal W_s)$ has exactly two poles. Hence each $X \setminus \mathcal N^X_{k}(\mathcal W_s)$ has at least one connected component not contained in any tubular neighbourhood of $\mathcal W_s$. In fact, since this component is not $s$-invariant, there are at least two such connected components of $X \setminus \mathcal N^X_{k}(\mathcal W_s)$. Since the number of poles of $(X,\mathcal W_s)$ equals two, all other possible connected components of $X \setminus \mathcal N^X_{k}(\mathcal W_s)$ must be contained in some tubular neighbourhood of $\mathcal W_s$. It remains to exclude the existence of these components.

It suffices to prove that any vertex $v$ of $X$ is adjacent to
some vertex $v'$ which is farther from $\mathcal W_s$. By
Lemma~\ref{lem:Reflection:basic}(ii), the reflection $s$ is not
contained in a conjugate of any spherical factor of $W$. Therefore
the desired statement follows from Corollary~\ref{lem:SphFactor}.
\end{proof}

\subsection{Reflections}

In this section we show that in a bipolar Coxeter group the notion of a \emph{reflection} is independent of the choice of a Coxeter generating set (Corollary~\ref{cor:refl}). It follows that all elements of all Coxeter generating sets are bipolar.

\begin{prop}\label{prop:refl}
Let $S$ be a Coxeter generating set for $W$ all of whose elements are bipolar. Then any involution of
$W$ which is not a reflection is unipolar.
\end{prop}

Proposition~\ref{prop:refl} is related to \cite[Corollary~2]{Klein} which asserts that, in an arbitrary finitely generated group, an infinite index subgroup of an $n$-polar subgroup is necessarily unipolar.

Before we provide the proof, we deduce the following corollary. We say that two Coxeter generating sets $S_1$ and $S_2$ for $W$ are \textbf{reflection-compatible} if every element of $S_1$ is conjugate to an element of $S_2$. This defines an equivalence relation on the collection of all Coxeter generating sets (see \cite[Corollary~A.2]{TwistRigid}).

\begin{cor}\label{cor:refl}
In a bipolar Coxeter group any two Coxeter generating sets are reflection-compatible.
In particular any Coxeter generating set consists of bipolar elements.
\end{cor}
\begin{proof}
By hypothesis there is some Coxeter generating set $S_1 \subset W$
consisting of bipolar elements. Let $r$ belong to an other Coxeter
generating set $S_2$. By Lemma~\ref{lem:Reflection:basic}(i)
$\centra_W(r)$ is not unipolar. Hence by
Proposition~\ref{prop:refl} the involution
$r$ is a reflection with respect to $S_1$.
\end{proof}

In order to prove Proposition~\ref{prop:refl} we need the
following subsidiary result. Let $d$ denote the maximal diameter
of a spherical residue in $X$.

\begin{lemma}\label{lem:Tubular}
Let $\mathcal W_1, \dots, \mathcal W_n$ be the walls associated to
the reflections of some finite parabolic subgroup $P < W$. Then
for each $k\in \N$ there is some $K\in \N$ satisfying
$$\bigcap_{i=1}^n \mathcal N^X_k(\mathcal W_i) \subset
\mathcal N^X_{K}\big(\bigcap_{i=1}^n \mathcal N^X _d(\mathcal
W_i)\big).$$
\end{lemma}

We need to consider the intersection of $\mathcal N^X_d(\mathcal
W_i)$ instead of the intersection of the walls $\mathcal W_i$
themselves because in the Cayley graph the intersection
$\bigcap_{i=1}^n \mathcal W_i$ is usually empty. On the other
hand, the intersection of $\mathcal N^X_d(\mathcal W_i)$ is
non-empty since it contains all the residues 
%PE replaced "of $P$" by 
whose stabiliser is $P$.

\begin{proof}
It is convenient here to work with the Davis complex $\mathbb A$.
The complex $\A$ equipped with its path-metric is quasi-isometric
to the Cayley graph $X$. In the language of the Davis complex, we
need to show that for each $k\in \N$, there is some $K\in \N$
satisfying
$$\bigcap_{i=1}^n \mathcal{N}^\A_k(\mathcal{W}^\A_i)
\subset \mathcal{N}^\A_K(\bigcap_{i=1}^n\mathcal{W} ^\A_i).$$

The above intersection $\bigcap_{i=1}^n \mathcal W^\A_i$ equals to
the fixed-point set $\A^P$ of $P$ in $\A$. The centraliser
$\centra_W(P)$ acts cocompactly on $\mathbb A^P$.

Assume for a contradiction that there is some sequence $(x_j)$
contained in $\bigcap_{i=1}^n \mathcal N^\A_k(\mathcal W^\A_i)$
but leaving every tubular neighbourhood of $\mathbb A^P$. After
possibly translating the $x_j$ by the elements of $\centra_W(P)$,
we may assume that the set of orthogonal projections of the $x_j$
onto $\mathbb A^P$ is bounded. Let $\xi \in \partial_\infty
\mathbb A$ be an accumulation point of $(x_j)$. If we pick a
basepoint $o$ in $\mathbb A^P$, then the geodesic ray $[o, \xi)$
leaves every tubular neighbourhood of $\mathbb A^P$. On the other
hand, by assumption we have $x_j \in \mathcal N^\A_k (\mathcal
W^\A_i)$. This implies $\xi \in \partial_\infty(\mathcal W^\A_i)$
and hence $[o, \xi) \subset \mathcal W^\A_i$ for each $i$. Thus we
have $[o, \xi) \subset \bigcap_{i=1}^n  \mathcal W^\A_i = \mathbb
A^P$, contradiction.
\end{proof}

We are now ready for the following.

\begin{proof}[Proof of Proposition~\ref{prop:refl}]
Let $r \in W$ be an involution which is not a reflection. We need to
show that $\centra_W(r)$ is unipolar.

Let $P$ be the minimal parabolic subgroup containing $r$ and let
$\mathcal W_1, \dots, \mathcal W_n$ be the walls corresponding to
all the reflections in $P$. The centraliser of $r$ acts
cocompactly on $\bigcap_{i=1}^n \mathcal N^X_d(\mathcal W_i)$
which we will denote by $X^r$. Hence $X^r$ is at finite Hausdorff
distance from $\centra_W(r)\subset X^{(0)}$. Therefore, in view of
Remark~\ref{rem:olddefin}, it suffices to show that $(X,X^r)$ has
only one pole.

By Lemma~\ref{lem:Tubular} for each $k\in \N$ there exists $K\in
\N$ satisfying
$$\mathcal N^X_{k-d}(X^r) \subset \bigcap_{i=1}^n
\mathcal N^X_k(\mathcal W_i) \subset \mathcal N^X_K(X^r).$$ Hence
it suffices to prove that for each $k\in \N$ the set
\begin{equation}
\label{set}
X\setminus \big(\bigcap_{i=1}^n \mathcal N^X_k(\mathcal W_i)\big)
\end{equation}
is connected. If we denote by $\Phi$ the set of all half-spaces
bounded by $\mathcal W_i$ for some $i$, the set displayed
in~(\ref{set}) is equal to $$\bigcup_{\phi  \in \Phi} \phi
\setminus \mathcal N^X_k(\partial \phi).$$

Since $W$ is bipolar, the set $\phi \setminus \mathcal
N^X_k(\partial \phi)$ is connected for each $\phi \in \Phi$.
Moreover, by Lemma~\ref{lem:Reflection:basic}(ii), $W$ has no
spherical factor. Therefore, by Lemma~\ref{lem:Fuchsian}(ii) the
intersection $\phi \setminus \mathcal N^X_k(\partial \phi) \cap
\phi' \setminus \mathcal N^X_k(\partial \phi')$ is non-empty for
any two non-complementary half-spaces $\phi, \phi' \in \Phi$.
Finally, since $r$ is not a reflection, we have $n>1$ and hence
$\Phi$ does not consist of a single pair of complementary
half-spaces.

Hence $\bigcup_{\phi  \in \Phi}  \phi \setminus \mathcal N^X_k(\partial \phi)$ is connected, $(X,X^r)$ has only one pole, and $r$ is unipolar, as desired.
\end{proof}

\subsection{Rigidity}

Finally, we prove our rigidity result.

\begin{proof}[Proof of Theorem~\ref{thm:CoxGenSet}]
Let $W$ be a bipolar Coxeter group and let $S_1$ and $S_2$ be two Coxeter generating sets for $W$. By Corollary~\ref{cor:refl}, the sets $S_1$ and $S_2$ are reflection-compatible; moreover, both of them consist of bipolar elements.

Let $X_i$ be the Cayley graph associated with the generating set $S_i$ and let $\Psi_i$ be the corresponding set of half-spaces. We shall denote by $\mathcal W_{r,i}$ the wall of
$X_i$ associated with a reflection $r \in W$.

We need the following terminology. A \textbf{basis} is a set of half-spaces containing a given vertex $v$ bounded by walls adjacent to $v$. A pair of half-spaces $\{\alpha, \beta\} \subset \Psi_i$ is called \textbf{geometric} if $\alpha \cap \beta$ is a fundamental domain for the action on $X_i$ of the group $\la r_\alpha, r_\beta \ra$ generated by the corresponding reflections. If $\la r_\alpha, r_\beta \ra$ is finite, then this means that for each reflection $r \in \la r_\alpha, r_\beta \ra$, the set $\alpha \cap \beta$ lies entirely in one half-space determined by the wall $\mathcal W_{r,i}$.
If $\la r_\alpha, r_\beta \ra$ is infinite, then this means that  $\alpha\cap \beta, \alpha\cap - \beta, -\alpha\cap \beta$ are all non-empty but $-\alpha\cap -\beta$ is empty. Note that if $r_\alpha, r_\beta$ commute, then $\{\alpha, \beta\}$ is automatically geometric.

\medskip

In order to show that $S_1$ and $S_2$ are conjugate, it suffices to show that there are half-spaces in $\Psi_2$ bounded by $\mathcal W_{s,2}$, over $s\in S_1$, which form a basis. In view of the main theorem of H\'ee \cite{Hee} (see also \cite[Theorem~1.2]{HRT} or \cite[Section 1.6]{CM} for other proofs of the same fact), it suffices to prove the following. There are half-spaces in $\Psi_2$ bounded by $\mathcal W_{s,2}$, over $s\in S_1$, which are pairwise geometric.

\smallskip

Let $S^0_1$ be the union of those irreducible components of $S_1$
which are not pairs of non-adjacent vertices (giving rise to
$D_\infty$ factors). For a generator $s\in S_1$ outside $S^0_1$ we
consider the unique other element $t$ in the irreducible component
of $s$. Then the walls $\mathcal W_{s,2}, \mathcal W_{t,2}$ are
disjoint and there is a geometric choice of half-spaces in
$\Psi_2$ for this pair. Since all other elements of $S_1$ commute
with both $s$ and $t$, it remains to choose pairwise geometric
half-spaces in $\Psi_2$ for the elements of $S_1^0$.

\medskip

Now the main part of the proof can start . The identity on $W$ defines a quasi-isometry $f: X_1^{(0)} \to X_2^{(0)}$, which we extend to an invertible (possibly non-continuous) mapping on the entire $X_1$. By Sublemma~\ref{subl:quasiisometry} and by the fact that $\mathcal W_{r,i}$ are at bounded Hausdorff distance from $\centra _W(r)\subset X_i^{(0)}$, we have the following. For each $\alpha \in \Psi_1$ bounded by $\mathcal W_{r,1}$, there is a (unique) half-space $\alpha' \in \Psi_2$ bounded by $\mathcal W_{r,2}$ satisfying
$$f\big(\alpha \setminus \mathcal N^X_k(\mathcal W_{r, 1})\big)
\subset \alpha'$$
for some $k\in \N$. Therefore, the assignment $\alpha \mapsto \alpha'$ defines a $W$-equivariant
bijection $f' : \Psi_1 \to \Psi_2$.

Let $\Phi\subset \Psi_1$ be the set of half-spaces containing the
identity vertex and bounded by a wall of the form $\mathcal
W_{s,1}$ for some $s\in S^0_1$. Our goal is to show that the map
$f'$ maps every pair of half-spaces from $\Phi$ to a geometric
pair in $\Psi_2$. Let $\alpha\neq \beta$ belong to $\Phi$. Set
$\alpha' = f'(\alpha)$ and $\beta' = f'(\beta)$. For $k\in \N$ and
any pair $\{\rho, \delta\} \subset \Psi_i$, we set
$$C_i(\rho, \delta, k) = \rho \setminus \mathcal N^{X_i}_k(\mathcal W_{r_\rho,i}) \cap
\delta \setminus \mathcal N^{X_i}_k(\mathcal W_{r_\delta,i}).$$

\smallskip \noindent \textbf{Case where $\partial \alpha$ and $\partial\beta$ intersect.}
In this case we proceed by contradiction. If $\{\alpha', \beta'\}$ is not geometric, then there exists a
reflection $r \in \la r_\alpha, r_\beta \ra$ different from $r_\alpha$ and $r_\beta$ satisfying the following. If $\phi'$ and $-\phi'$ denote the pair of half-spaces in
$\Psi_2$ bounded by $\mathcal W_{r,2}$, then both $\alpha'\cap\phi'$ and $\beta'\cap-\phi'$ are non-empty and contained in $\alpha'\cap \beta'$.

By Lemma~\ref{lem:Fuchsian}(ii) for all $k\in \N$ both
$C_2(\alpha', \phi', k)$ and $C_2(\beta', -\phi',k)$ are
non-empty. Denote $f'^{-1}(\phi')=\phi$. We now apply
Sublemma~\ref{subl:quasiisometry} to $f^{-1}$. It guarantees that
for $k$ large enough the sets $C_2(\alpha', \phi', k)$ and
$C_2(\beta', -\phi',k)$ are mapped into $\alpha\cap\phi$ and
$\beta\cap-\phi$, respectively. Furthermore, they are both mapped
into $\alpha\cap\beta$. Hence $\alpha\cap \beta$ is separated by
the wall $\mathcal W_{r,1}$ and $\{\alpha, \beta\}$ is not
geometric. Contradiction.

\smallskip \noindent \textbf{Case where $\partial \alpha$ and $\partial\beta$ are disjoint.}
By Lemma~\ref{lem:Fuchsian}(i,iii) all the sets  $C_1(\alpha,\beta,k), C_1(\alpha, -\beta,k),$ and $C_1(-\alpha, \beta,k)$ are non-empty. Hence all $\alpha'\cap\beta', \alpha'\cap -\beta',$ and $-\alpha'\cap \beta'$ are non-empty. This means that $\{\alpha',\beta'\}$ is geometric.

\end{proof}

%%%%%%%%%%%%%%%%%%%%%%%%%%%%%%%%%%%%%%%%%%%%%%%%
%%%%%%%%%%%%%%%%%%%%%%%%%%%%%%%%%%%%%%%%%%%%%%%%
\section{Characterisation of nearly bipolar reflections}
%%%%%%%%%%%%%%%%%%%%%%%%%%%%%%%%%%%%%%%%%%%%%%%%
%%%%%%%%%%%%%%%%%%%%%%%%%%%%%%%%%%%%%%%%%%%%%%%%
\label{sec:CharNearlyBiv}

On our way to proving Theorem~\ref{thm:bipolar}, which characterises bipolar Coxeter groups, we come upon a property slightly weaker than bipolarity, which we discuss in this section.

Given a vertex $v$ in the Cayley graph $X$ and a reflection $r\in W$, we denote by
$\mathcal C_{v, r}$ the subset of $X$ which is the intersection of half-spaces containing $v$ bounded by walls orthogonal to or equal $\mathcal W_r$. Note that $\mathcal C_{v, r}$ is a fundamental domain for the action on $X$ of the group generated by reflections in these walls. We say that $r$ is \textbf{nearly bipolar} if for all $k\in \N$ and each vertex $v$ of $X$, the set $\mathcal C_{v,r} \setminus \mathcal N^X_k(\mathcal W_r)$ is non-empty and connected.

\smallskip
The goal of this section is to prove the following (for the notation, see Section~\ref{sec:preliminaries}).

\begin{thm}\label{thm:NearlyBipolarReflection}
Let $r \in W$ be a reflection. The following assertions are equivalent.
\begin{enumerate}[(i)]
\item $r$ is nearly bipolar.

\item The following two conditions are satisfied by every vertex $v\in X$.
\begin{enumerate}[a)]
\item $T_{v, r}$ is not a spherical irreducible component of $S_v$.

\item $J_{v, r} \cup U_{v, r}$ does not separate $S_v$.
\end{enumerate}
\end{enumerate}
\end{thm}

Below we prove that for a bipolar or nearly bipolar reflection $r$ there are ways to connect a pair of walls by a chain of walls avoiding tubular neighbourhoods of $\mathcal W_r$. The proof bears resemblance to the main idea of \cite{TwistRigid}, where to obtain isomorphism rigidity we had to connect a pair of \emph{good markings} by a chain of other markings with \emph{base} $r$.

\begin{lemma}\label{lem:3refl}
Let $s, t\in S$ be non-adjacent and let $r \in W$ be a reflection. Assume that at least one of the following conditions is satisfied.
\begin{enumerate}[(i)]
\item
$r$ is nearly bipolar and $\la s, t\ra$ does not contain any reflection commuting with $r$.
\item
$r$ is bipolar and at most one reflection from  $\la s, t\ra$ commutes with $r$. This reflection does not equal $r$.
\end{enumerate}
Then for any $k\in \N$ there is a sequence of reflections
$s=r_0, r_1, \dots, r_n=t$
such that for all $i =1,
\dots ,n$ the wall $\mathcal W_{r_{i-1}}$ intersects $\mathcal W_{r_i}$ and for all $i = 1, \dots, n-1$ the wall $\mathcal W_{r_i}$ is disjoint from $\mathcal N^X_k(\mathcal W_r)$.
\end{lemma}

\begin{proof}
Denote by $v_0 \in X$ the identity vertex. Without loss of generality, we may assume that the given $k$ is larger than the distance from $v_0$ to $\mathcal W_r$.
By Corollary~\ref{cor:ParallelWall}, there is a constant
$L$ such that for any vertex $v$ at distance at least $L$ from
$\mathcal W_r$, there is a wall separating $v$ from $\mathcal
N^X_k(\mathcal W_r)$.

\begin{figure}[ht]
\includegraphics[width=11cm]{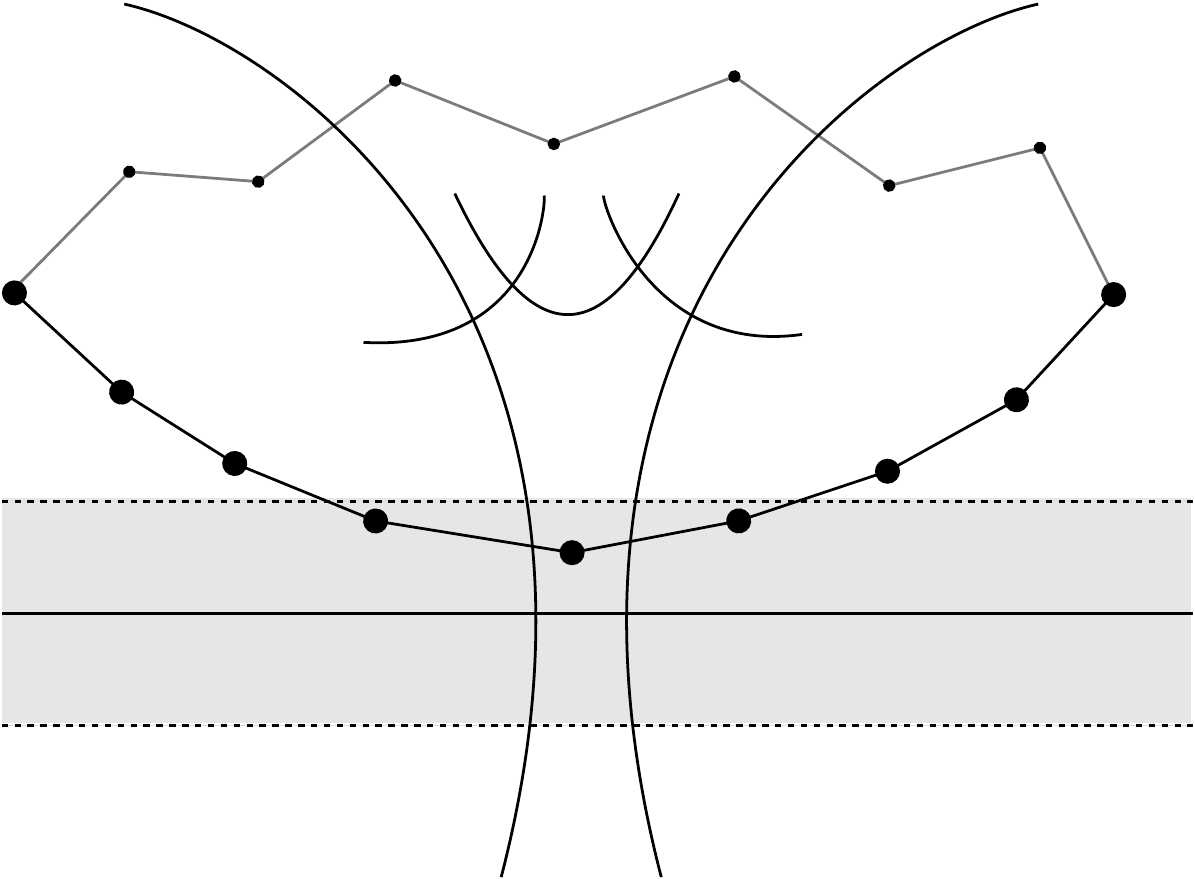}
\put(2, 67){$\mathcal W_r$}
\put(-280, 52){$\mathcal N^X_k(\mathcal W_r)$}
\put(-200, 12){$\mathcal W_s$}
\put(-137, 12){$\mathcal W_t$}
\put(-30, 130){$R$}
\put(-300, 150){$v_-$}
\put(-40, 150){$v_+$}
\put(-217, 217){$x_1$}
\put(-175, 200){$x_2$}
\put(-130, 217){$x_{n-1}$}
\put(-235, 140){$\mathcal W_{1}$}
\put(-168, 130){$\mathcal W_{2}$}
\put(-101, 140){$\mathcal W_{n-1}$}
\caption{Proof of Lemma~\ref{lem:3refl}}
\label{fig:Lemma42}
\end{figure}

Denote by $R$ be the $\{s, t\}$-residue containing $v_0$ (see
Figure~\ref{fig:Lemma42}). Since $r$ does not belong to $\la
s,t\ra$, the residue $R$ lies entirely on one side of $\mathcal
W_r$. We claim that  $\la r, s,t \ra$ is a hyperbolic triangle
group. Indeed, $\la r, s, t\ra$ is an irreducible reflection
subgroup of rank $3$, hence by \cite{Deodhar_refl} it is a Coxeter
group of rank $3$. Since it contains an infinite parabolic
subgroup of rank~$2$, namely $\la s, t\ra$, it cannot be of affine
type. Thus $\la r, s, t\ra$ is a hyperbolic triangle group, as
claimed. The claim implies that every tubular neighbourhood of
$\mathcal W_r$ contains at most a bounded subset of the residue
$R$.

Hence for $N$ large enough, the vertices $v_- = (st)^{-N}.v_0$ and
$v_+ = (st)^{N}.v_0$ are not contained in $\mathcal N^X_L(\mathcal
W_r)$. By hypothesis, either $r$ is nearly bipolar and the vertices $v_-$ and $v_+$ are both contained in $\mathcal C_{v_0,r}$ or $r$ is bipolar. Thus there is a path connecting $v_+$ to $v_-$ outside $\mathcal N^X_L(\mathcal W_r)$.

There is a sub-path $(x_1, \dots, x_{n-1})$ of $\gamma$ such that $x_1$ is adjacent to $\mathcal W_s$ and $x_{n-1}$ is adjacent to $\mathcal W_t$. By the choice of $L$, for each $i$ there is some wall which separates $x_i$ from
$\mathcal N^X_k(\mathcal W_r)$. Among these, we pick one nearest possible $\mathcal W_r$ and call it $\mathcal W_i$. We denote the associated reflection by $r_i$.

Notice first that, since $\mathcal W_1$ separates $x_1$ from $v_0$, which are both adjacent to $\mathcal W_s$, it follows that $\mathcal W_1$ intersects $\mathcal W_s$. Analogously $\mathcal W_{n-1}$ intersects $\mathcal W_t$. It remains to show that
$\mathcal W_{i-1}$ intersects $\mathcal W_{i}$
for all $i = 2, \dots, n-1$.

Assume for a contradiction that $\mathcal W_{i-1}$ does not intersect $\mathcal W_{i}$. In particular we have $\mathcal W_{i-1} \neq \mathcal W_{i}$ and it follows that for some $j \in \{i-1, i\}$, say for $j=i$, the vertices $x_{i-1}$ and $x_i$ lie on the same side of $\mathcal W_{j}$. It follows that the vertex $x_{i-1}$ is separated from $\mathcal N^X_k(\mathcal W_r)$ by both $\mathcal W_{i-1} $ and $\mathcal W_{i}$. By the minimality hypothesis on $\mathcal W_{i-1}$, the wall $\mathcal W_{i-1}$ separates $\mathcal N^X_k(\mathcal W_r)$  from $\mathcal W_i$. But this contradicts the minimality hypothesis on $\mathcal W_i$.
\end{proof}

We can now provide the proof of the main result of this section.

\begin{proof}[Proof of Theorem~\ref{thm:NearlyBipolarReflection}]
(i) $\Rightarrow$ (ii)
Assume that $r$ is nearly bipolar. Since $\mathcal C_{v,r}\setminus \mathcal N^X_k(r)$ is non-empty for each $k\in \N$, the set $X\setminus \mathcal N^X_k(r)$ is non-empty for each $k$. Then $\tilde e(W, \centra_W(r))$ is non-zero and in view of Lemma~\ref{lem:Reflection:basic}(ii) we have condition a).

\medskip

It remains to prove condition b), which we do by contradiction.
Assume that there are $s, t \in S_v$ separated by $J_{v,r}\cup
U_{v,r}$. We set $J = J_{v, r},\ U = U_{v, r},$ and $T=T_{v,r}$.
By Lemma~\ref{lem:polygon} the group  $\la s, t\ra$ does not
contain any reflection which commutes with $r$. Therefore, we are
in position to apply Lemma~\ref{lem:3refl}(i). Let $k$ be large
enough so that the residue stabilised by $W_{J\cup (U\cap T)}$ and
containing $v$ (this residue is finite by Lemma~\ref{lem:shadow})
lies entirely in $\mathcal N^X_k(\mathcal W_r)$.
Lemma~\ref{lem:3refl}(i) provides a sequence of reflections
$s=r_0, \dots, r_n = t$ such that for all $i=1, \dots, n-1$ the
wall $\mathcal W_{r_i}$ avoids $\mathcal N^X_k(\mathcal W_r)$ and
for all $i=1, \dots, n$ walls $\mathcal W_{r_{i-1}}$ and $\mathcal
W_{r_i}$ intersect.

The group $W$ splits over $W_{J \cup U}$ as an amalgamated product
of two factors each containing one of $s$ and $t$. Consider now
the $W$-action on the associated Bass--Serre tree $\mathcal T$.
Thus $W_{J \cup U}$ is the stabiliser of some edge $e$ of
$\mathcal T$, and the elements $s$ and $t$ fix distinct vertices
of $e$, but neither of them fixes $e$. Furthermore, for each $i =
1, \dots, n$, the fixed-point sets $\mathcal T^{r_{i-1}}$ and
$\mathcal T^{r_i}$ intersect. It follows that some $r_i$ fixes the
edge $e$, hence it lies in $W_{J \cup U}$. From the inclusions $J
\subset T,\ U\subset T\cup T^\perp$ and $T^\perp\subset U$, we
deduce $$W_{J \cup U} = W_{J\cup (U\cap T)} \times W_{T^\perp}.$$
Thus a reflection in $W_{J\cup U}$ belongs either to $W_{J\cup
(U\cap T)}$ or to $W_{T^\perp}$. Since the wall $\mathcal W_{r_i}$
does not meet $\mathcal W_r$, the order of $rr_i$ must be
infinite, hence $r_i$ does not belong to $W_{T^\perp}$. Therefore
we have $r_i\in W_{J\cup(U\cap T)}$. This implies that $\mathcal
W_{r_i}$ meets the residue stabilised by $W_{J\cup(U\cap T)}$
containing $v$, contradicting the fact that $\mathcal W_{r_i}$
avoids $\mathcal N^X_k(\mathcal W_r)$.

\medskip

\item
(ii) $\Rightarrow$ (i)
Let $k\in \N$ and let $v \in X$ be a vertex. We need to show that
$$\mathcal C_{v,r} \setminus \mathcal N^X_k(\mathcal W_r)$$
is non-empty and connected. For non-emptiness it suffices to prove that any vertex $w$ of $X$ is adjacent to a vertex which is farther from $\mathcal W_r$. Otherwise we have $S_w=J_{w,r}\cup U_{w,r}$ and it follows that $S_v$ equals $T_{w,r}\cup T_{w,r}^\perp$. Moreover, $T_{w,r}$ is then equal to $J_{w,r}\cup (U_{w,r}\cap T_{w,r})$, which is finite by Lemma~\ref{lem:shadow}. This would contradict condition a).

It remains to prove connectedness. Let $x, y$ be two vertices in $\mathcal C_{v,r} \setminus \mathcal N_k^X(\mathcal W_r)$. We shall construct a path connecting $x$ to $y$ outside of $\mathcal N^X_k(\mathcal W_r)$. First notice that, by the definition of $\mathcal C_{v,r}$, no wall orthogonal to $\mathcal W_r$ separates $x$ from $y$.

We consider the collection $\mathcal G$ of all (possibly
non-minimal) paths connecting $x$ to $y$ entirely contained in $\mathcal C_{v,r}$. Notice that $\mathcal G$ is non-empty since it contains all minimal length paths from $x$ to $y$. To each path $\gamma \in \mathcal G$,
we associate a $k$-tuple of integers $(n_1, \dots, n_k)$,
where $n_i$ is defined as the number of vertices of $\gamma$ at distance $i-\frac{1}{2}$ from $\mathcal W_r$. We call this tuple $(n_1, \dots, n_k)$ the
\textbf{trace} of the path $\gamma$. We order the elements of
$\mathcal G$ using the lexicographic order on the set of their traces.

We need to show that $\mathcal G$ contains some path of trace
$(0, \dots, 0)$. To this end, it suffices to associate to every
path in $\mathcal G$ with non-zero trace a path of strictly
smaller trace. Let thus $\gamma \in \mathcal G$ be a path with non-zero trace
$(n_1, \dots, n_k)$, put $j = \min \{i \; | \; n_i
>0\}$ and let $v$ be some vertex of $\gamma$ contained in
$\mathcal N^X_j(\mathcal W_r)$. Let also $v_-$ and $v_+$ be
respectively the predecessor and the successor of $v$ on $\gamma$. The vertices
$v_-$ and $v_+$ do not belong to $\mathcal N^X_j(\mathcal W_r)$ (otherwise $\gamma$ would cross walls which are orthogonal to
$\mathcal W_r$). Set $J = J_{v, r}, \ T = T_{v, r},$ and $U = U_{v, r}$. Let $s_-$ and $s_+$ be the elements of $S_v$ satisfying $v_- = s_-.v$  and $v_+=s_+.v$. Since $v_-$ and $v_+$ do not belong to $\mathcal N^X_j(\mathcal W_r)$, we infer that $s_-$ and $s_+$ do not belong to
$J\cup U$.

Condition b) implies existence of a path
$$s_-= s_0, s_1, \dots, s_m = s_+$$
connecting $s_-$ to $s_+$ in $S_v \setminus (J \cup U)$. Put
$v_k = s_k. v$ for
$k = 0, 1, \dots, m$. In particular $v_0 = v_-$ and $v_m = v_+$. Notice that
for each $k=1,\ldots, m$ the rank-$2$ residue containing $v$ and stabilised by  $\la s_{k-1}, s_k\ra$ is finite. Therefore, it contains a path $\gamma_k$
connecting $v_{k-1}$ to $v_k$ but avoiding $v$. Since $s_{k-1}$ and $s_{k}$ are not in $J \cup U$, we deduce
from Lemma~\ref{lem:polygon} that $\gamma_k$ does not intersect
$\mathcal N^X_j(\mathcal W_r)$, and that no wall crossed by $\gamma_k$
is orthogonal to $\mathcal W_r$.

We now define a new path
$\gamma' \in \mathcal G$ as follows. The path $\gamma'$ coincides
with $\gamma$ everywhere, except that the sub-path $(v_-, v, v_+)$
is replaced by the concatenation $\gamma_1 \dots
\gamma_m$. Notice that $\gamma'$ is entirely contained in $\mathcal C_{v,r}$. Denoting the trace of $\gamma'$ by $(n'_1, \dots,
n'_k)$, it follows from the construction that we have $n'_i = 0$ for all
$i <j$ and $n'_j < n_j$. Hence the trace of $\gamma'$ is
smaller than the trace of $\gamma$, as desired.
\end{proof}

%%%%%%%%%%%%%%%%%%%%%%%%%%%%%%%%%%%%%%%%%%%%%%%%
%%%%%%%%%%%%%%%%%%%%%%%%%%%%%%%%%%%%%%%%%%%%%%%%
\section{Characterisation of bipolar reflections}
%%%%%%%%%%%%%%%%%%%%%%%%%%%%%%%%%%%%%%%%%%%%%%%%
%%%%%%%%%%%%%%%%%%%%%%%%%%%%%%%%%%%%%%%%%%%%%%%%
\label{sec:bipolar reflections}

In this section we finally prove Theorem~\ref{thm:bipolar}. We deduce it from Theorem~\ref{thm:BipolarReflection} characterising bipolar reflections, which is similar in spirit to Theorem~\ref{thm:NearlyBipolarReflection}. In order to state it we introduce the following terminology.

Given two reflections $r, t \in W$, we say that {$r$ \textbf{dominates} $t$} (or $t$ \textbf{is dominated by} $r$) if the wall $\mathcal W_t$ is contained in some tubular neighbourhood of $\mathcal W_r$.
In particular, $t$ is dominated by $r$ if $\centra_W(t)$ is virtually contained in $\centra_W(r)$ (the converse is also true, but we do not need it).

\begin{thm}\label{thm:BipolarReflection}
Let $r \in W$ be a reflection. The following assertions are equivalent.
\begin{enumerate}[(i)]
\item $r$ is bipolar.

\item $r$ is nearly bipolar and does not dominate any reflection $t\neq r$ commuting with $r$.

\item The following three conditions are satisfied by every vertex $v$ of $X$.
\begin{enumerate}[a)]
\item $T_{v, r}$ is not a spherical irreducible component of $S_v$.

\item There is no non-empty spherical $I\subset T_{v,r}$ such that $I\cup T_{v, r}^\perp$ separates $S_v$.

\item If  $T_{v, r}$ is spherical and an odd component $O$ of $S_v$ is contained in $T_{v, r}^\perp$, then there are adjacent $t \in O$ and $t' \in S_v \setminus (T_{v, r} \cup T_{v, r}^\perp)$.
\end{enumerate}
\end{enumerate}
\end{thm}

Before providing the proof of Theorem~\ref{thm:BipolarReflection}, we apply it to the following.

\begin{proof}[Proof of Theorem~\ref{thm:bipolar}]
First assume that $W$ is bipolar, \emph{i.e.} for some Coxeter generating set $S\subset W$ all elements of $S$ are bipolar. Given any irreducible subset $T \subset S$, there exists a reflection $r \in W_T$ with \emph{full support}, \emph{i.e.} a reflection which is not contained in $W_{T'}$ for any proper subset $T' \subset T$. Let $v_0$ denote the identity vertex of $X$. Then we have $T = T_{v_0, r}$. Conditions a), b), and c) of Theorem~\ref{thm:bipolar} follow now directly from conditions a), b), and c) of Theorem~\ref{thm:BipolarReflection}.

Conversely, assume that $S\subset W$ satisfies conditions a), b), and c) of Theorem~\ref{thm:bipolar}. Since for any $v,r$ the set $T_{v,r}$ is irreducible, these yield immediately conditions a), b) and c) of Theorem~\ref{thm:BipolarReflection}. Hence every reflection of $W$ is bipolar and $W$ is bipolar.
\end{proof}

We begin the proof of Theorem~\ref{thm:BipolarReflection} with a (probably well-known) lemma which indicates the role of the odd components.

\begin{lemma}\label{lem:OddComp}
Let $s \in S$, let $O$ be the odd component of $s$ in $S$ and let $\bar O$ be the set of all elements of $S$ adjacent to some element of $O$. Then the centraliser $\centra_W(s)$ is contained in $W_{\bar O}$.
\end{lemma}
\begin{proof}
Consider an element $w$ of the centraliser $\centra_W(s)$. Denote
by $v_0$ the identity vertex in $X$. By \cite[Proposition
5.5]{Deodhar} there is a sequence of vertices $v_0, v_1,\ldots,
v_n=w.v_0$, such that all $v_i$ are adjacent to $\mathcal W_s$ and
the pairs $v_{i-1},v_i$ lie in a rank-$2$ residue $R_i$
intersecting $\mathcal W_s$. Denote by $s_i\in S$ the type of the
edge between $v_i$ and $s.v_i$, in particular we have $s_0=s$. We
can show inductively that if $R_i$ is of type $\{s_{i-1},t\}$ with
$s_{i-1}$ and $t$ odd-adjacent, then $s_i$ equals $t$. If
$s_{i-1}$ and $t$ are not odd-adjacent, then $s_i$ equals
$s_{i-1}$. It follows that $w.v_0$ is connected to $v_0$ by a path
of edges all of whose types lie in $\bar O$.
\end{proof}

\begin{proof}[Proof of Theorem~\ref{thm:BipolarReflection}]
We first provide the proof of the less involved equivalence (i) $\Leftrightarrow$ (ii). Then we give the proofs of (i) $\Rightarrow$ (iii) and of (iii) $\Rightarrow$ (ii).

\medskip
\noindent (i) $\Rightarrow$ (ii)
Assume that $r$ is bipolar. Then clearly $r$ is nearly bipolar. Consider a reflection $t \neq r$ commuting with $r$ and let $k\in \N$. Since $r$ is bipolar, there is a vertex $v$ lying outside $\mathcal N^X_k(\mathcal W_r)$. In particular $v' = t. v$ is another such vertex and moreover $v$ and $v'$ lie on the same side of $\mathcal W_r$. Since $r$ is bipolar, there is a path joining $v$ to $v'$ outside
$\mathcal N^X_k(\mathcal W_r)$. This path must cross $\mathcal W_t$, hence $\mathcal W_t$ is not contained in  $\mathcal N^X_k(\mathcal W_r)$, as desired.

\medskip \noindent
(ii) $\Rightarrow$ (i)
Assume now that $r$ is nearly bipolar and does not dominate any reflection $t\neq r$ commuting with $r$. Let $k\in \N$ and let $x, y$ be vertices of $X$ outside of $\mathcal N^X_k(\mathcal W_r)$ not separated by $\mathcal W_r$. Let $\mathcal W_1, \dots, \mathcal W_n$ be all the walls orthogonal to $\mathcal W_r$ which are successively crossed by some minimal length path joining $x$ to $y$. For each $i$, since the reflection in $W_i$ is not dominated by $r$, we can pick a pair of adjacent vertices $z_i, z'_i$ lying outside of $\mathcal N^X_k(\mathcal W_r)$ and such that $z_i$ (resp.\ $z'_i$) lies on the same side of $\mathcal W_i$ as $x$ (resp.\ $y$). Denote additionally $z'_0 = x$ and $z_{n+1} = y$. Since $r$ is nearly bipolar, any two vertices outside of $\mathcal N^X_k(\mathcal W_r)$ and not separated by any wall orthogonal to $\mathcal W_r$ may be connected by a path lying entirely outside of $\mathcal N^X_k(\mathcal W_r)$. Thus for each $i = 0, \dots, n$ there is a path avoiding $\mathcal N^X_k(\mathcal W_r)$ and connecting $z'_i$ to $z_{i+1}$. Concatenating all these paths we obtain a path avoiding  $\mathcal N^X_k(\mathcal W_r)$ and joining $x$ to $y$. This shows that $r$ is bipolar, as desired.

\medskip
This ends the proof of equivalence (i) $\Leftrightarrow$ (ii). It
remains to prove the equivalence with (iii).

\medskip \noindent
(i) $\Rightarrow$ (iii)
We assume that $r$ is bipolar. Like in the proof of Theorem~\ref{thm:NearlyBipolarReflection}, condition a) follows from Lemma~\ref{lem:Reflection:basic}(ii).

\smallskip

We now prove condition b), by contradiction. Suppose that there is a vertex $v$ and non-empty spherical $I \subset T_{v,r}$ such that $I\cup T^\perp_{v,r}$ separates some $s, t \in S_v$ in the Coxeter diagram of $S_v$. In particular, the group $\la s, t\ra$ is infinite. We set $T=T_{v,r}$.

\begin{claim}
The group $\la s, t\ra$ contains at most one reflection commuting with $r$. This reflection does not equal $r$.
\end{claim}

In order to establish the claim, we first notice that $r$ does not belong to $\la s, t\ra$. Otherwise we would have $I \subset T \subset \{s,t\}$, which is impossible since neither $s$ nor $t$ belongs to $I$ and $I$ is non-empty.

In particular, the rank-$2$ residue $R$ stabilised by $\la s,
t\ra$ and containing $v$ lies entirely on one side of $\mathcal
W_r$. Let $v'$ be a vertex in $R$ at a minimal distance to
$\mathcal W_r$ ($v'$ might be not uniquely determined) and let
$s'$ and $t'$ denote the two reflections of $\la s,t\ra$ whose
walls are adjacent to $v'$.

If at most one of $s', t'$ commutes with $r$, then by
Lemma~\ref{lem:polygon} this is the only reflection of $\la s', t'\ra = \la s, t\ra$ commuting with $r$, as desired. On the other hand, if $s'$ and $t'$ both commute with $r$, then $r$ centralises $\la s, t\ra$. By \cite[Proposition~5.5]{Deodhar}, this implies that $r$ belongs to the parabolic subgroup $\la \{s, t\}^\perp \ra$. By definition, $T\subset S_v$ is smallest such that $r$ is contained in $W_ T$. We infer that $T$ is contained in $\{s, t\}^\perp$, or equivalently that $s$ and $t$ lie in $T^\perp$. This contradiction ends the proof of the claim.

\medskip

In view of the claim, we are in a position to apply Lemma~\ref{lem:3refl}(ii). It provides for each $k\in \N$ a sequence of reflections $s=r_0, \dots, r_n = t$ such that for all $i =1,
\dots ,n$ the wall $\mathcal W_{r_{i-1}}$ intersects $\mathcal W_{r_i}$ and for all $i = 1, \dots, n-1$, the wall $\mathcal W_{r_i}$ avoids $\mathcal N^X_k(\mathcal W_r)$.
We now consider the $W$-action on the Bass--Serre tree associated
with the splitting of $W$ over $W_{I \cup T^\perp}$ as an amalgamated product of two factors containing $s$ and $t$, respectively. We obtain a contradiction using the exact same arguments as in the proof of Theorem~\ref{thm:NearlyBipolarReflection}((i)$\Rightarrow$(ii)).

\smallskip
It remains to prove condition c), which we also do by contradiction. Assume that there is a vertex $v$ of $X$ such that $T=T_{v,r}$ is spherical, an odd component $O$ of $S_v$ is contained in $T^\perp$ and no pair of elements of $O$ and $S_v\setminus T\cup T^\perp$, respectively, is adjacent. Denote by $\bar O$ the union of $O$ with the set of all elements of $S_v$ adjacent to an element of $O$. Pick any $s\in O$.

By Lemma~\ref{lem:OddComp}, the centraliser $\centra_W(s)$ is contained in $W_{\bar O}$, which is in our case contained in $W_{T \cup T^\perp}$. Then, since $T$ is spherical, the group $\centra_W(s) \cap W_{T^\perp}$ has finite index in $\centra_W(s)$. On the other hand, clearly $W_{T^\perp}$ is contained in $\centra_W(r)$. Therefore, we deduce that $\centra_W(s)$ is virtually contained in $\centra_W(r)$, which implies that $r$ dominates $s$.
Contradiction.

\medskip \noindent
(iii) $\Rightarrow$ (ii)
By Lemma~\ref{lem:shadow}, the set $I=J_{v,r}\cup (T_{v,r} \cap U_{v,r})$ is spherical, for any vertex $v$ of $X$. Hence, by Theorem~\ref{thm:NearlyBipolarReflection}, conditions a) and b) imply that $r$ is nearly bipolar.

It remains to prove that there is no reflection $t\neq r$ dominated by $r$, which we do by contradiction. If there is such a $t$, then let $v$ be a vertex adjacent to $\mathcal W_t$ at maximal possible distance from the wall $\mathcal W_r$. We again set $J = J_{v, r}$,  $T = T_{v,r},$ and $U = U_{v, r}$. We have $t\in U\subset S_v$. To proceed we need the following general remark. Its part (i) requires Lemma~\ref{lem:polygon}.

\begin{rema} Let $s\in S_v$ be adjacent to $t$ and let $m$ denote the order of $st$. Put $v'=(st)^{\lfloor\frac{m}{2}\rfloor}.v$.
\begin{enumerate}[(i)]
\item
For $s\not\in J\cup U$ we have $d(v',\mathcal W_r)>d(v,\mathcal W_r)$ and $v'$ is adjacent to $\mathcal W_t$.
\item
For $s\in U$ we have $d(v',\mathcal W_r)=d(v,\mathcal W_r)$. 
Moreover the canonical bijections between $S_v$, $S$ and $S_{v'}$ yield identifications  $T_{v',r}\cong T,\ J_{v',r}\cong J,$ and $U_{v',r}\cong U$. We denote by $s_0$ the element of $S$ corresponding to $s\in S_v$, \emph{i.e.} such that $v$ and $s.v$ share an edge of type $s_0$. If $m$ is odd, then $v'$ is adjacent
to $\mathcal W_t$ by an edge of type $s_0$.
\end{enumerate}
\end{rema}

The proof splits now into two cases.

\smallskip \noindent \textbf{Case $t \in T^\perp$.} In this case we have $T=J\cup (U\cap T)$, since otherwise $v$ is adjacent to another vertex adjacent to $\mathcal W_t$ farther away from $\mathcal W_r$. Hence $T$ is spherical by Lemma~\ref{lem:shadow}.

By part (i) of the Remark, $t$ is not adjacent to any element
outside $T\cup T^\perp$. In particular, every element $s$
odd-adjacent to $t$ lies in $T^\perp$. Then, by part (ii) of the
remark, we can replace $v$ with $v'$, which replaces in the
Coxeter graph the vertex corresponding to $t$ with the one
corresponding to $s$. Hence the whole odd component of $s$ is
contained in $T^\perp$ and none of its elements is adjacent to a
vertex outside $T\cup T^\perp$. This contradicts condition c).

\smallskip \noindent \textbf{Case $t\not \in T^\perp$.} In this case we set
$$I = J \cup (T \cap U)  \setminus \{t\}.$$
By Lemma~\ref{lem:shadow} the set $I \cup \{t\}$ is spherical, in
particular so is $I$. Observe that $I\cup \{t\}\cup T^\perp$ does
not equal the whole $S_v$. Indeed, otherwise we would have
$S_v=T\cup T^\perp$ with $T = I \cup \{t\}$ spherical which
contradicts condition a).

By condition b) the set $I \cup T^\perp$ does not separate $S_v$.
Therefore, there exists some $s\in S_v \setminus (I \cup T^\perp)$
adjacent to $t$. By part (i) of the Remark this leads to a
contradiction.
\end{proof}

We finish this section with an example of a Coxeter group
all of whose reflections are nearly bipolar, but not all are bipolar.

\begin{example}\label{ex}
Let $(W,S)$ be the Coxeter group associated with the Coxeter diagram represented in Figure~\ref{fig:ex}, where each solid edge is labeled by the Coxeter number $4$, while each dotted edge is labeled by the Coxeter number $2$. In particular, the pair $\{s_2, s_6\}$ is non-spherical.

It follows easily from Theorem~\ref{thm:NearlyBipolarReflection}
that every reflection of $W$ is nearly bipolar. On the other hand,
put $r=s_1$ and let $v_0$ be the identity vertex. Then we have
$T_{v_0,r}=\{s_1\}$. The singleton $\{s_6\}$ is an odd component
contained in $T_{v_0,r}^\perp$. But $s_6$ is not adjacent to the
only element outside $T_{v_0,r}\cup T_{v_0,r}^\perp$, which is
$s_2$. This violates condition c) of
Theorem~\ref{thm:BipolarReflection}(iii). Hence $s_1$ is not
bipolar.

We can see explicitly that Proposition~\ref{prop:refl} fails for
$W$. Consider the subset $S' = \{s'_1, \dots, s'_6\} \subset W$
defined by $s'_i = s_i$ for all $i<6$ and $s'_6 = s_1 s_6$.
Clearly $S'$ is a generating set consisting of involutions.
Moreover each pair $\{s'_i, s'_j\}\subset S'$ satisfies the same
relations as the corresponding pair $\{s_i, s_j\} \subset S$.
Therefore the mapping $s_i \mapsto s'_i$ extends to a well-defined
surjective homomorphism $\alpha \colon W \rightarrow W$. Since $W$
is finitely generated and residually finite, it is Hopfian by
\cite{Malcev}. Thus $\alpha$ is an automorphism and $S'$ is a
Coxeter generating set. But $s'_6$ is not a reflection, 
which confirms that the conclusions of Proposition~\ref{prop:refl} do not hold in this example. 
\end{example}

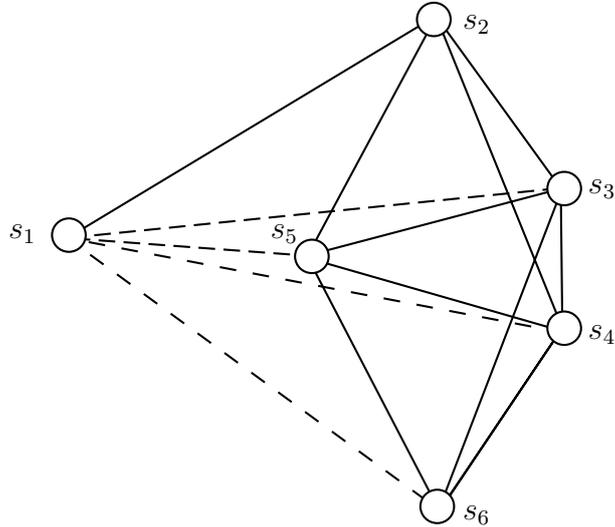
\begin{figure}[ht]
\includegraphics[width=7cm]{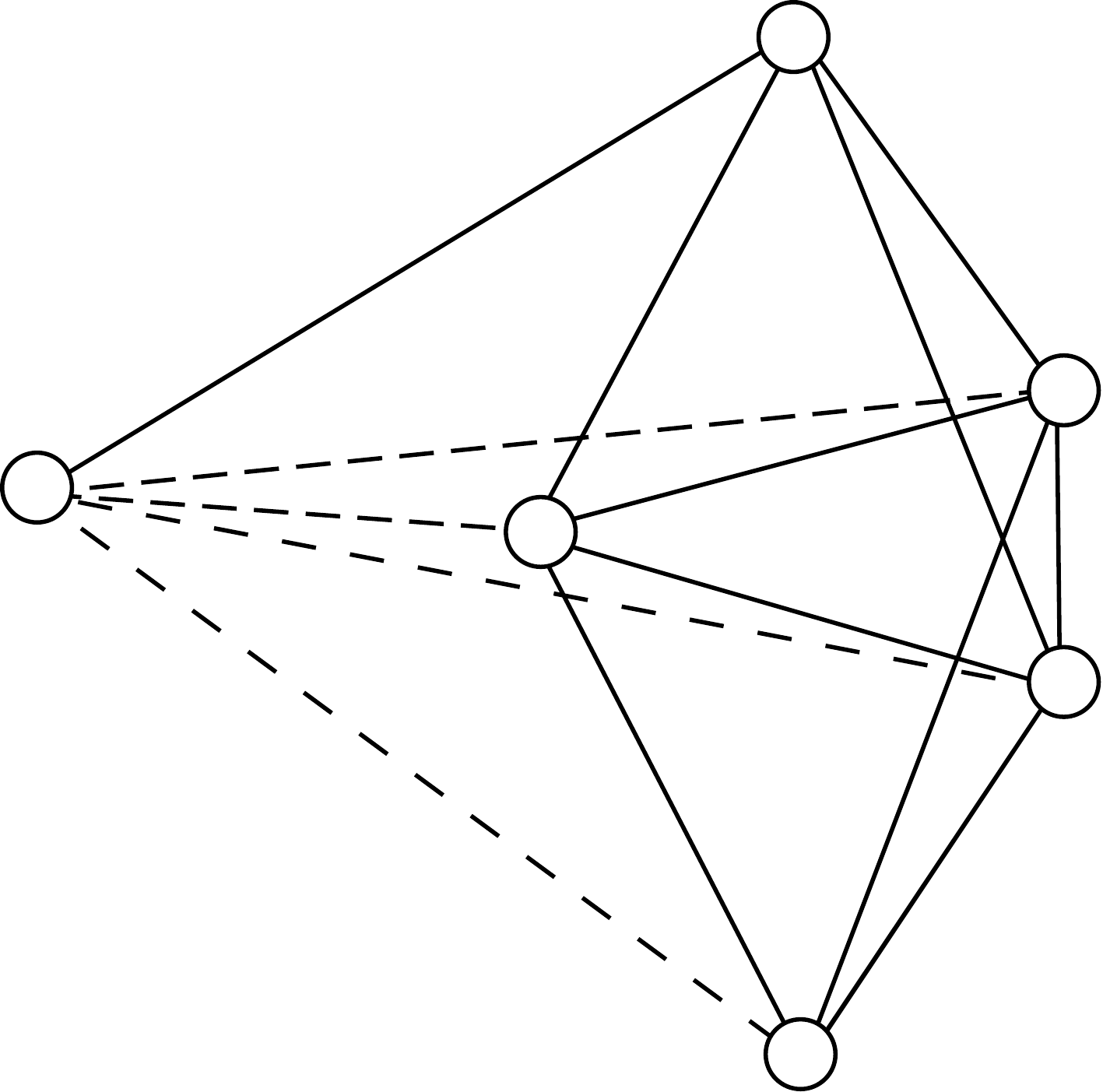}
\put(-215,108){$s_1$}
\put(-45,2){$s_6$}
\put(-45,187){$s_2$}
\put(2,125){$s_3$}
\put(2,70){$s_4$}
\put(-117,108){$s_5$}
\caption{Coxeter diagram for Example~\ref{ex}}
\label{fig:ex}
\end{figure}

\appendix

%%%%%%%%%%%%%%%%%%%%%%%%%%%%%%%%%%%%%%%%%%%%%%%%
%%%%%%%%%%%%%%%%%%%%%%%%%%%%%%%%%%%%%%%%%%%%%%%%
\section{Poles}
\label{App}
%%%%%%%%%%%%%%%%%%%%%%%%%%%%%%%%%%%%%%%%%%%%%%%%
%%%%%%%%%%%%%%%%%%%%%%%%%%%%%%%%%%%%%%%%%%%%%%%%

This appendix is aimed at a discussion of the notion of a pole in a general framework.

\subsection{Poles}

Let $H$ be a subset of a metric space $X$. A \textbf{pole} of $X$
relative to $H$ (or of the pair $(G, H)$) is a chain of the form
$U_1 \supset U_2 \supset \dots$, where $U_k$ is a non-empty
connected component of $X \setminus \mathcal N^X_k(H)$.

\smallskip

A different but equivalent definition of a pole is as follows. Let
$\mathscr{H}$ denote the collection of subsets of $X$ at bounded
Hausdorff distance from $H$ and let $\mathscr P (X)$ be the set of
all subsets of $X$. A pole of $X$ relative to $H$ (or of the pair
$(X, H)$) is a function $U\colon \mathscr{H} \rightarrow \mathscr
P (X)$ satisfying the following two conditions, where $H_1, H_2
\in \mathscr H$:
\begin{itemize}
\item $U(H_1)$ is a non-empty connected component of $X\setminus H_1$.
\item  If  $H_1\subset H_2$, then $U(H_1)\supset U(H_2)$.
\end{itemize}
This equivalent definition makes the following remark obvious.

\begin{rem}\label{rem:olddefin}
Let $H_1,H_2\subset X$ be at finite Hausdorff distance. Then we can identify the poles of $(X,H_1)$ with the poles of $(X,H_2)$.
\end{rem}

We now prove that poles are quasi-isometry invariants.

\begin{lemma}\label{lem:quasiisometry}
Let $X$ and $Y$ be two path-metric spaces and let $f\colon
X\rightarrow Y$ be a quasi-isometry. Then there is a natural
correspondence between the poles of $(X,H)$ and the poles of
$(Y,f(H))$.
\end{lemma}

In order to prove Lemma~\ref{lem:quasiisometry} we will establish the following.

\begin{sublem}\label{subl:quasiisometry}
Let $f\colon X\rightarrow Y$ be a quasi-isometry between a metric space $X$ and a path-metric space $Y$.
Then for each $k\in \N$ there is $K\in \N$ such that for each connected component $\alpha$ of $X \setminus \mathcal N^X_K(H)$, there is a connected component $\alpha'$ of $Y \setminus \mathcal N^Y_k(f(H))$ satisfying
$$f\big( \alpha\big) \subset \alpha'.$$
\end{sublem}

Before we provide the proof of Sublemma~\ref{subl:quasiisometry}, we show how to use it in the proof of the lemma.

\begin{proof}[Proof of Lemma~\ref{lem:quasiisometry}]
Let $V_1 \supset V_2 \supset \dots$ be a pole of the pair $(X,H)$.
We define its corresponding pole $U_1 \supset U_2 \supset \dots$
of $(Y,f(H))$. By Sublemma~\ref{subl:quasiisometry}, for each
$k\in \N$ there is a component $U_k$ of $Y \setminus \mathcal
N^Y_k(f(H))$ which contains the $f$-image of some $V_{K(k)}$.
Since all $V_{K(k)}$ intersect, for $k'>k$ we have $U_k\supset
U_{k'}$. Thus $U_1 \supset U_2 \supset \dots$ is a pole. Hence we
have a mapping $f'$ from the collection of poles of $(X,H)$ to the
collection of poles of $(Y,f(H))$. We now prove that $f'$ is a
bijection.

Let $g\colon Y \rightarrow X$ be a quasi-isometry which is
quasi-inverse to $f$. Let $g'$ be the map induced by $g$ which
maps the collection of poles of $(Y,f(H))$ to the collection of
poles of $(X,g\circ f(H))$. The sets $H$ and $g\circ f(H)$ are at
finite Hausdorff distance and by Remark~\ref{rem:olddefin} we can
identify the poles of $(X,g\circ f(H))$ with the poles of $(X,H)$.
We leave it to the reader to verify that $f'\circ g'$ and $g'\circ
f'$ are the identity maps. Thus $f'$ is a bijection.
\end{proof}

It remains to prove the sublemma.

\begin{proof}[Proof of Sublemma~\ref{subl:quasiisometry}]
We need the following terminology. Given $k\in \N$, a sequence $(x_0, \dots, x_n)$ of points in $X$ is called a \textbf{$k$-path} if the
distance between any two consecutive $x_i$'s is at most $k$. A subset
$Z \subset X$ is called \textbf{$k$-connected} if any two elements
of $Z$ may be joined by some $k$-path entirely contained in $Z$.

Let $c$ and $L$ be the additive and the multiplicative constants of the quasi-isometry $f$. Put $K=L(k+L+2c)$. Then $f(X\setminus \mathcal N^X_K(H))$ is contained in $Y\setminus \mathcal N^Y_{k+L+c}(f(H))$.

Let $\alpha$ be a connected component of $X\setminus \mathcal N^X_K(H)$. The quasi-isometry $f$ maps $\alpha$, which is $1$-connected, to an $(L+c)$-connected subset of $ Y\setminus \mathcal N^Y_{k+L+c}(f(H))$. Any pair of points at distance $L+c$ in $Y\setminus \mathcal N^Y_{k+L+c}(f(H))$ is connected by a path in $Y$ of length at most $2(L+c)$ (here we use the hypothesis that $Y$ is a path-metric space). This path has to lie in $Y\setminus \mathcal N^Y_k(f(H))$. Hence the points of any connected component $\alpha$ of $X\setminus \mathcal N^X_K(H)$ are mapped into a single connected component of $Y\setminus \mathcal N^Y_k(f(H))$.
\end{proof}

We conclude with the following alternative characterisation of
poles. A subset of $X$ is called \textbf{$H$-essential} if it is
not contained in any tubular neighbourhood of $H$.

\begin{lemma}\label{lem:essential components}
\begin{enumerate}[(i)]
\item Suppose that the number of poles of $(X,H)$ is finite and equals $n$. Then for $k$ sufficiently large the number of connected $H$-essential components of the space $X \setminus \mathcal N^X_k(H)$ is exactly $n$.
\item
On the other hand, if the number of poles of $(X,H)$ is infinite, then for $k$ sufficiently large the number of connected $H$-essential components of the space $X \setminus \mathcal N^X_k(H)$ is arbitrarily large.
\end{enumerate}
\end{lemma}

We leave the proof as an exercise to the reader.

\subsection{Poles as topological ends}
\label{sec:topological ends}

It is natural to ask if the poles of $(X, H)$ may be identified with the topological ends of a certain space. Below we construct such a topological space $X_{\widehat H}$ which, as a set, coincides with the disjoint union of $X$ together with one additional point, denoted by $\infty$. The topology on $X_{\widehat H}$ is defined in the following way. First, we declare that the embedding $X \to X_{\widehat H}$ is continuous and open. Second, we define neighbourhoods of $\infty$ to be complements of those subsets of $X$ which intersect every tubular neighbourhood of $H$ in a bounded subset. In particular, if $H$ is bounded, then $\infty$ is an isolated point.

If $X$ is locally compact, there is an alternative approach. For each $k\in \N$ there is a natural continuous embedding
$$\widehat{\mathcal N^X_k(H)} \to X_{\widehat H},$$
where we denote by $\widehat Z$ the one-point compactification of a space $Z$.
In view of this, the space $X_{\widehat H}$ can be alternatively defined as the direct limit of the injective system given by the natural continuous embeddings $\big\{\widehat{\mathcal N^X_k(H)} \to \widehat{ \mathcal N^X_{k'}(H)}\big\}_{k<k'}.$

\begin{lemma}\label{lem:compact}
For any compact subset $Q \subset X_{\widehat H}$, the intersection $X \cap Q$ is contained in some tubular neighbourhood of $H$.
\end{lemma}

\begin{proof}
Let $Q \subset X_{\widehat H}$ be a subset which contains a sequence $(x_k)$ of $X$ such that $x_k$ does not belong to $\mathcal N^X_k(H)$. Clearly $(x_k)$ is unbounded in $X$. Moreover, the complement of the set $\{x_k\}_k$ is a neighbourhood of $\infty$, so that $(x_k)$ does not sub-converge to $\infty$ in $X_{\widehat H}$. This implies that $Q$ is not compact.
\end{proof}

Lemma~\ref{lem:compact} implies that a sequence $(x_k)$ in $X$ converges to $\infty$ if and only if it leaves every bounded subset of $X$ but remains in some tubular neighbourhood of $H$. The lemma also immediately implies the following.

\begin{prop}\label{prop:DefBipolar}
There is a natural correspondence between the poles of $(X,H)$ and the topological ends of $X_{\widehat H}$.
\end{prop}

\subsection{Poles in groups}

Let now $G$ be a finitely generated group and let $X$ denote the Cayley graph associated with some finite generating set for $G$. We view $X$ as a path-metric space with edges of length~$1$. We identify $G$ with the $0$-skeleton $X^{(0)}$ of $X$. Let $H$ be a subset of $G$.

We recall that if $H$ is a subgroup, then $e(G, H)$ denotes the number of \textbf{relative ends} of $G$ with respect to $H$, which are the topological ends of the quotient space $H\backslash X$. This invariant was first introduced by Houghton~\cite{Houghton} and Scott~\cite{Scott} and is independent of the choice of a generating set for $G$.

On the other hand, we define a \textbf{pole} (in $X$) of $G$ relative to $H$ (or of the pair $(G, H)$) to be a pole of $(X,H)$. By Lemma~\ref{lem:quasiisometry}, there is a correspondence between the collections of poles of the pair $(G,H)$ determined by different generating sets. Hence we can speak about the number of poles of $(G,H)$, which we denote by $\tilde e(G, H)$. Here $H$ is allowed to be any subset of $G$.

By Proposition~\ref{prop:DefBipolar}, we have a correspondence between the poles of $(X,H)$ and the ends of the space $X_{\widehat H}$. In particular, by Lemma~\ref{lem:quasiisometry}, there is natural correspondence between the ends of $X_{\widehat H}$ and the ends of $Y_{\widehat H}$, where $Y$ is the Cayley graph of $G$ with respect to a different generating set.

Our notation $\tilde e(G, H)$ for the number of poles coincides with the notation of
Kropholler and Roller~\cite{KR}. Their definition goes as follows.

Let $\mathscr P G$ denote the set of all subsets of $G$ and $\mathscr F_H G$ the collection of all subsets of $G$ contained in $HF$ for some finite subset $F$ of $G$. Notice that an element of $\mathscr F_H G$ is nothing but a subset of $G$ lying in some tubular neighbourhood of $H$ in the Cayley graph. We view $\mathscr P G$ and $\mathscr F_H G$ as vector spaces over the field $\FF_2$ of order two.

The action of $G$ on itself by right multiplication preserves both $\mathscr P G$ and  $\mathscr F_H G$; they can thus be viewed as right $G$-modules over $\FF_2$. Kropholler and Roller set
\begin{equation}
\label{Krop}
\tilde e(G, H) = \dim_{\FF_2} (\mathscr P G/\mathscr F_H G)^G.
\end{equation}
See also Geoghegan \cite[Section IV.14]{G} for a similar definition of this value, which is called there the number of \emph{filtered ends}. We end the appendix by establishing the following.

\begin{lemma}\label{lem:tildeecoincides}
The number of poles of $(G,H)$ coincides with the value $\tilde e(G, H)$ defined by the formula~(\ref{Krop}).
\end{lemma}
\begin{proof}
If the number of poles of $(G,H)$ is at least $n$, then there is $k\in \N$ such that $X\setminus \mathcal N_k^X(H)$ has at least $n$ connected $H$-essential components (see Lemma~\ref{lem:essential components}).
The set of vertices of each such component determines a non-trivial vector of $(\mathscr P G/\mathscr F_H G)^G$. Moreover, the collection of all these vectors is linearly independent. This implies $\tilde e(G, H)\geq n$.

Conversely, let $v_1,\ldots, v_n$ be linearly independent vectors in $(\mathscr P G/\mathscr F_H G)^G$. Let $V_i$ be the subset of $X^{(0)}$ determined by $v_i$. Denote by $\partial V_i$ the set of all the vertices outside $V_i$ which are adjacent to some vertex in $V_i$. Then all $\partial V_i$ are at finite Hausdorff distance from $H$. Choose $k\in \N$ so that $\mathcal N_k^X(H)$ contains all $\partial V_i$. Then each $v_i$ lies in the linear subspace of $(\mathscr P G/\mathscr F_H G)^G$ determined by the connected $H$-essential components of $X\setminus \mathcal N_k^X(H)$.
Hence $n$ is bounded by the number of connected $H$-essential components of $X\setminus \mathcal N_k^X(H)$, which equals at most $\tilde e(G, H)$.
\end{proof}

\end{document}